\documentclass{article}
\usepackage{amsmath, amssymb, verbatim, textcomp, graphicx, framed}
\usepackage{graphicx}
\usepackage[a4paper,left=2.5cm,right=2.5cm,top=3cm,bottom=2cm]{geometry}
\newtheorem{theorem}{Theorem}

\newtheorem{corollary}{Corollary}

\newenvironment{proof}{{\sc Proof:}}{~\hfill$\Box$}
\newtheorem{remark}{Remark}
\def\e{\mathrm{e}}
\def\Z{\mathbb{Z}}
\def\N{\mathbb{N}}
\def\ra{\rightarrow}
\def\R{\mathbb{R}}
\def\L{\widehat{\ell}}
\begin{document}
\title{On the moments of roots of Laguerre-polynomials and the Marchenko-Pastur law}



\author{\Large{M. Kornyik}  \\  E\"otv\"os Lor\'and University \\ Department of Probability Theory and 
Statistics \\ P\'azm\'any P\'eter s\'et\'any 1/C., H-1117, Budapest, Hungary \\
              \textit{email}:koma@cs.elte.hu        \and
        \Large{Gy. Michaletzky} \\   E\"otv\"os Lor\'and University \\ Department of Probability Theory and 
Statistics \\ P\'azm\'any P\'eter s\'et\'any 1/C., H-1117, Budapest, Hungary\\
              \textit{email}:michaletzky@caesar.elte.hu}
\maketitle

\begin{abstract}
        In this paper we compute the leading terms in the sum of the $k^{th}$ power of the roots of
$L_{p}^{(\alpha)}$, the Laguerre-polynomial of degree $p$ with parameter $\alpha$.
The connection between the Laguerre-polynomials and the Marchenko-Pastur distribution is expressed
by the fact, among others, that
the limiting distribution of the empirical distribution of the normalized roots of the Laguerre-polynomials is given by
the Marchenko-Pastur distribution. We give a direct proof of this statement based on the recursion satisfied
by the Laguerre-polynomials. At the same time,
our main result gives that the leading term in $p$ and $(\alpha+p)$ of the sum of the $k^{th}$ power of the roots of
$L_{p}^{(\alpha)}$ coincides with the $k^{th}$ moment of the Marchenko-Pastur law.
We also mention the fact that the expectation of the characteristic polynomial of a $XX^T$ type random covariance matrix,
where $X$ is a $p\times n$ random matrix with iid elements, is $\ell^{(n-p)}_p$, i.e. the monic version of the $p^{th}$ Laguerre polynomial with parameter $n-p$.
\end{abstract}

\section{Introduction}

In theory of orthogonal polynomials the limit of the empirical distribution of their roots is a much studied matter.
In this paper we are going to study the limit distribution of the roots of Laguerre polynomials $L_p^{(\alpha_p)}$, where
\begin{align}\label{Lag_coeff}
    L_p^{(\alpha)}(x)=\sum_{j=0}^p{(-1)^j\binom{\alpha+p}{p-j}\frac{x^j}{j!}} \ \ \ \ \ \alpha\in\R
\end{align}
assuming that $\alpha_p/p\ra c>-1$.
For $\alpha>-1$ these polynomials are known to be orthogonal with respect
to the measure $x^{\alpha}e^{-x}\textbf{1}_{ [0,\infty]} dx$, from which one can conclude that
all the roots are distinct and lie in $\R_+$. For $\alpha\in [-p+1,-1]\cap\Z$ one has that
$$ L_p^{(\alpha)}(x)=x^{-\alpha}L_{p+\alpha}^{(-\alpha)}(x) $$
and hence one can make the conclusion that for such $\alpha$ values the polynomial
$L_p^{(\alpha)}$ has $p+\alpha$ disticnt positive roots and $0$ is also a root with multiplicity $-\alpha$.

In section 1 we show that the normalized generating function of the moments of the normalized roots of
$L_p^{(\alpha_p)}$ satisfies
the same quadratic fixed point equation in the limit as the generating function of the moments of the Marchenko-Pastur 
distribution.

In section 2 we will explicitly show that the coefficient of the highest order term (viewed as a polynomial in $p$) of the $k^{th}$ power of the roots of $L_p^{(\alpha)}$ coincides with the $k^{th}$ moment of the corresponding
Marchenko-Pastur distribution.

\section{Convergence of the empirical distribution}

Let us consider the roots of the Laguerre-polynomial $L_p^{(\alpha)}$ denoted by $\xi_{p,1}^{(\alpha)}, \dots , \xi_{p,p}^{(\alpha)}$.
Let $M^{(\alpha)}_p(k)$ denote the
sum of their $k$-th power. 
That is $M^{(\alpha)}_p(k) = \sum_{i=1}^p (\xi_{p,i}^{( \alpha)})^k$.
Finally, $\mathcal M_p^{(\alpha)}$ denotes the power series
determined by these coefficients, i.e.
\begin{equation} \label{mgf}
\mathcal{M}^{(\alpha)}_p(z) = p+\sum_{k=1}^\infty M^{(\alpha)}_p(k) z^k\,.
\end{equation}
Note that in case $\alpha$ is a negative integer in the interval $[-p+1,-1]$ the zero is also a root of $L_p^{(\alpha)}$, 
which explains why the case $k=0$, i.e. the zeroth moment, had to be dealt with seperately in (\ref{mgf}).
 It is known that
\begin{align*}
   \mathcal{M}^{(\alpha)}_p(z)=\frac1z \frac{(\ell_{p}^{(\alpha)})'(1/z)}{\ell_p^{(\alpha)}(1/z)}=-z\frac{(\widehat{\ell}_{p}^{(\alpha)})'(z)}{\widehat{\ell}_{p}^{(\alpha)}(z)}+p
\end{align*}
    where $ \ell^{(\alpha)}_p(x)=(-1)^p p! L_p^{(\alpha)}(x)$ is the monic version of $L_p^{(\alpha)}$,
and for any polynomial of degree $p$ we denote by
$\L(z)=z^p \ell(1/z)$ the so-called conjugate polynomial.

\begin{theorem}\label{MP_weaklimit}
Let us assume that $\alpha = \alpha_p $ and $\frac{\alpha_p}{p}\rightarrow c\in (-1,\infty)$, as $p\rightarrow \infty$.
Then the empirical distribution determined by the normalized roots (where \ $p^{-1}\ $ is the normalization factor) of the
 Laguerre-polynomial $L_p^{(\alpha_p)}$ converges weakly to the
Marchenko-Pastur distribution, given as
\begin{align}\label{MP_distr}\mu_c(A)=\begin{cases}
-c\delta_0(A)+ \nu_c(A)\,, & \mbox{ if } -1< c<0 \mbox{ and } \alpha_p\in\{-p+1,\ldots,-1\} \ \hbox{  for all  } p\,,\\
\nu_c(A)\,, & \mbox{ if } c\geq 0,
\end{cases}
\end{align}
for $ A \in \mathcal{B}(\R)$,
where $\delta_0$ denotes the Dirac-delta measure at $0$, while the measure $\nu_c$ is absolutely \\ continuous with density
$$ d\nu_c(x)= \frac{\sqrt{(x_+-x)(x-x_-)}}{2\pi x}\mathbf{1}_{[x_-,x_+]}(x)dx $$
where $x_\pm=[ \sqrt{c+1}\pm1]^2.$
\if 0\textbf{Megj.: Szerintem a fenti helyett az k\'ene, hogy}
$$ d\nu_c(x)=\frac{\sqrt{(x-x_-')(x_+'-x)}}{2\pi x}\textbf{1}_{[x_-',x_+']}(x)dx $$
\textbf{ahol} $x_\pm'=(\sqrt{c+1}\pm1)^2$. \fi
\end{theorem}
\begin{remark}

 A more general version of this theorem -- allowing for $c < -1$ --  was proved by Mart\'inez-Gonz\'alez et al. in \cite{Mar01} using complex analysis and differential equations, but the proof presented here is based on elementary calculations using only the recursion equations
satisfied by the Laguerre-polynomials.
\end{remark}

\begin{remark}
Laguerre polynomials show a deep connection with random matrix theory in the following ways:
\begin{itemize}
 \item[1.] Forrester and Gamburd proved in \cite{for06} that the expectation of the characteristic polynomial 
of the random matrix $XX^T$ is given by $\ell_p^{(n-p)}(z)$, i.e. $E \det (x\cdot I -XX^T)=\ell_p^{(n-p)}(x)$, 
where $X$ is a $p\times n$ random matrix with independent, identically distributed entries with zero expectation and variance $1$.
\item[2.]  If $X$ is a $p\times n$ random matrix in the same sense as above, 
then the weak limit of the empirical measure of the eigenvalues is a much studied question  of random matrix theory, although it is 
usually normalized by $n$,  which in our case means a normalization by $\alpha+p$. 
A well-known theory -- proved by Marchenko and Pastur in \cite{mar67} -- states that the weak limit of the empirical measure of the eigenvalues of
$\frac1n XX^T$ is given by $\tilde{\mu}_a$ as $\frac p n \ra a > 0$, where $\tilde{\mu}_a$ is defined below. In the case of the
present paper $\mu_c$ is the weak limit of the empirical measure of the eigenvalues of $\frac1p XX^T$.

\item[3.] The matrix theoretical Marchenko-Pastur distribution with parameter $a>0$ is given by
$$ \tilde{\mu}_a(A)=\begin{cases} \left(1-\frac1a\right) \delta_0 + \tilde{\nu}_a(A) , & \mbox{ if } a\in(0,1) \\ 
\tilde{\nu}_a(A) & \mbox{ if } a\geq 1\end{cases}  \ \ \ A\in\mathcal{B}(\R). $$
with $\tilde{\nu}_a$ being absolutely continuous with density
$$ d\tilde{\nu}_a(x)= \frac{\sqrt{(x-\tilde{x}_-)(\tilde{x}_+-x)}}{2\pi  ax}\mathbf{1}_{[\tilde{x}_-,\tilde{x}_+]}(x)dx $$
where $\tilde{x}_\pm=(1\pm\sqrt{a})^2$.
As mentioned before this version of the Marchenko Pastur arises when the zeros of $\ell_p^{(\alpha_p)}(z)$ are 
normalized by a factor of $(p+\alpha_p)^{-1}$. The connection between $d\mu_c$ and $d\tilde{\mu}_a$ is the following: 
\begin{eqnarray}
	a&=&\frac1{c+1}\,, \\
	\mu_c&=& \tilde{\mu}_a\circ g^{-1} \,.
\end{eqnarray}
where $g(x)=(c+1)x$ for $x\in\R$.
On the other hand it is known that the moments of $\tilde{\mu}_a$ are given by 
\begin{equation}
	\int x^k d\tilde{\mu}_a(x)=\sum_{j=1}^k \frac 1 k\binom{k}{j}\binom{k}{j-1} a^{j-1},
\end{equation}
hence the moments of $\mu_c$ can be calculated as
$$ \int x^{k}d\mu_c=\sum_{j=1}^{k}\frac1k\binom{k}{j}\binom{k}{j-1}(c+1)^{k-j+1} .$$
\end{itemize}
\end{remark}

\begin{proof}[of the Theorem]

Let $\ell^{(\alpha)}_p(z):=(-1)^pp!L_p^{(\alpha)}(z)$ denote the monic version of $L_p^{(\alpha)}(z)$ then
\begin{align}\label{monic_lag_coeff}
\ell_p^{(\alpha)}(z)=\sum_{j=0}^{p}{(-1)^j \frac{(p)_j(\alpha+p)_j}{j!}z^{p-j}}
\end{align} with $(\beta)_k=\beta(\beta-1)\cdots(\beta-k+1)$ for $k>0$.
Note that if $\beta$ is a positive integer and $k>\beta$ then $(\beta)_k=0$.
Thus it follows from (\ref{monic_lag_coeff})
that for $\alpha \in \{ -p+1, -p+2, \dots , -2, -1\}$
\begin{equation}\label{Lag_neg_par}
    \ell_p^{(\alpha)}(z)=z^{-\alpha} \ell_{p+\alpha}^{(-\alpha)}(z)\,.
\end{equation}
This means that in this case zero is a root of $\ell_p(x)^{(\alpha)}(z)$ with multiplicity $-\alpha$ and the other $p+\alpha$ roots given by the Laguerre-polynomial  $L_{p+\alpha}^{(-\alpha)}(z)$ of degree $p+\alpha$.

\begin{itemize}
\item[1.] Let us first consider the case when $\alpha_p \geq 0$ for all $p$, which also implies $\lim \alpha_p / p = c \geq 0$.

The recursion of the Laguerre-polynomials for arbitrary parameter $\alpha>-1$ is
\begin{equation}\label{Lag_recursion}
a_pL_{p+1}^{(\alpha)}(z)=(b_p-z)L_p^{(\alpha)}(z)-c_pL_{p-1}^{(\alpha)}(z),
\end{equation}
where $a_p=p+1$, $b_p=2p+\alpha+1$ and $c_p=p+\alpha$ and
also
\begin{eqnarray}\label{Lag_recursion_2}
pL_p^{(\alpha)}(z)& = & (p+\alpha)L_{p-1}^{(\alpha)}-z L_{p-1}^{(\alpha+1)}(z) \,
\end{eqnarray}
\if 0\begin{equation}
    L_p^{(\alpha)}(x)=\sum_{j=0}^p{(-1)^j\binom{p+\alpha}{p-j}\frac{1}{j!}x^j}.
\end{equation} \fi
These polynomials are known to be orthogonal with respect to the measure $z^\alpha \e^{-z}\mathbf{1}_{[0,\infty)}(z)dz$,
which implies that all the roots of $L_p^{(\alpha)}(x)$ lie in the interval $[0,\infty)$ and
hence the sum of the $k^{th}$ power of its roots is positive. Furthermore
\begin{align} \label{Lag_diff}
\frac{d}{dz}L_p^{(\alpha_p)}(z)= -L_{p-1}^{(\alpha_p+1)}(z) \,.
\end{align}
implying, after proper algebraic transformations, that
\begin{equation}
	\frac{d}{dz}\L^{(\alpha)}_p(z) = -(\alpha+p)p \L^{(\alpha)}_{p-1}(z)
\end{equation}
where $\L_p^{(\alpha)}(z)=z^p\ell_p^{(\alpha)}(z^{-1})$.
Applying this for $\alpha = \alpha_p$ we obtain that
\begin{equation}\label{f_M_kapcs}
	\frac1p\mathcal{M}^{(\alpha_p)}_p\left(\frac{z}{p}\right) =  \frac {\alpha_p+p} p  \frac{z\L_{p-1}^{(\alpha_p)}(z/p)}{\L_p^{(\alpha_p)}(z/p)}+1
\end{equation}
Also from recursion (\ref{Lag_recursion}) we get
\begin{align}
\label{conjlag-rec}
\L_{p+1}^{(\alpha)}(z)=[1-(\alpha+2p+1)z]\L_{p}^{(\alpha)}(z)-z^2(p+\alpha)p\L_{p-1}^{(\alpha)}(z).
\end{align}
Since the largest zero of $L_p^{(\alpha)}$ is no greater then $4p + 2\alpha +3$ (see \cite{sze39}) we obtain that
$\L_p^{(\alpha)}(z)> 0$, if $0\leq z < \frac{1}{4p+2\alpha + 3}$.


In this case one has that
$$ \frac{\L_{p-1}^{(\alpha)}(z)}{\L_{p}^{(\alpha)}(z)}\leq  \frac{1-(\alpha+2p+1)z}{z^2p(\alpha+p)}\leq \frac {1}{(p+\alpha)pz^2}. $$

Using the computations above we get that
\begin{align} \label{diff_bd}
	\frac d {dz} \frac{\L_{p-1}^{(\alpha)}(z)}{\L_p^{(\alpha)}(z)}&=
\frac{(\L_{p-1}^{(\alpha)}(z))'\L_p^{(\alpha)}(z)-(\L_p^{(\alpha)}(z))'\L_{p-1}^{(\alpha)}(z)}{(\L_p^{(\alpha)}(z))^2}= \nonumber\\
	 &=(\alpha+p)p\left[ \left(\frac{\L_{p-1}^{(\alpha)}(z)}{\L_p^{(\alpha)}(z)}\right)^2-\frac{\L_{p-2}^{(\alpha)}(z)}{\L_p^{(\alpha)}(z)} \right]
+(\alpha+2p-1)\frac{\L_{p-2}^{(\alpha)}(z)}{\L_p^{(\alpha)}(z)}.
\end{align}
 Since in the present case $\mathcal{M}^{(\alpha)}_p(z)$ is a convex, monotonically increasing function for $z\geq 0$,
and furthermore $\mathcal{M}^{(\alpha)}_p(0)=1$, one has
\begin{align}\label{int-bd}
z\frac{d}{dz}\frac{\L_{p-1}^{(\alpha)}(z)}{\L_p^{(\alpha)}(z)}\leq
\int_0^{2z}{\frac{d}{dt}\frac{\L_{p-1}^{(\alpha)}(t)}{\L_p^{(\alpha)}(t)}dt}=\frac{\L_{p-1}^{(\alpha)}(2z)}{\L_p^{(\alpha)}(2z)}-1\leq \frac{1}{4(\alpha+p)pz^2} \end{align}
and so according to (\ref{diff_bd}) and to (\ref{int-bd}) we have
\begin{equation}\label{bdd}
\left|\left( \frac{\L_{p-1}^{(\alpha)}(z/p)}{\L_p^{(\alpha)}(z/p)}\right)^2-\frac{\L_{p-2}^{(\alpha)}(z/p)}{\L_p^{(\alpha)}(z/p)}\right|\leq \frac{p^3}{4(\alpha+p)^2p^2z^3}+\frac{\alpha+2p-1}{(\alpha+p)^2p^2}\frac{p^4}{(p-1)(\alpha+p-1)z^4}. \end{equation}
Let $f^{(\alpha)}_p(z):=\frac{\L_{p-1}^{(\alpha)}(z/p)}{\L_p^{(\alpha)}(z/p)}$, for $p\geq 1$.
According to (\ref{bdd}) we have
\[
f^{(\alpha)}_p(z)f^{(\alpha)}_{p-1}\left(z\frac{p-1}p\right)-(f^{(\alpha)}_p)^2(z)\ra 0
\]
if
$p\ra \infty$ and $\alpha=\alpha_p \geq 0$, especially
$f^{(\alpha_p)}_p(z)f^{(\alpha_p)}_{p-1}\left((p-1)z/p)\right)-(f^{(\alpha_p)}_p)^2(z)\ra 0$ if $\frac {\alpha_p} {p} \ra c $
as $p\ra \infty$.

Applying (\ref{conjlag-rec}) to
$\L_{p}^{(\alpha)}$ one has
\begin{eqnarray}\label{f_rec}
1& = & \left(1-z\frac{\alpha+2p-1}{p}\right)f^{(\alpha)}_p(z) - \nonumber \\ & & -z^2\frac{(p-1)(\alpha+p-1)}{p^2}f^{(\alpha)}_p(z)f^{(\alpha)}_{p-1}\left(z\frac{p-1}{p}\right)\,,
\end{eqnarray}

hence
we get that the accumulation points of $(f^{(\alpha_p)}_p(z))_{p\in\N}$ as $\alpha_p/p \ra c$ satisfy the following 
equation in $\xi$
\begin{equation}\label{fix_eq} 1= \left[1-(c+2)z\right]\xi-(c+1)z^2 \xi^2. \end{equation}
The solutions of this equation are
$$ \xi_{\pm}=\frac{1-(c+2)z\pm \sqrt{[1-(c+2)z]^2-4(c+1)z^2}}{2(c+1)z^2}\,. $$
Let us introduce the notation
$$ f_{c-}(z):=\frac{1-(c+2)z- \sqrt{[1-(c+2)z]^2-4(c+1)z^2}}{2(c+1)z^2}. $$
In order to find the appropriate root let us look at the map $\xi\ra \eta_c(\xi, z)$  for a fixed $z$ defined by
$$ 1= \left[1-(c+2)z\right]\eta_c(\xi, z)-(c+1)z^2 \xi\eta_c(\xi, z) $$
and hence
$$ \eta_c(\xi, z)= \frac{1}{1-(c+2)z-(c+1)z^2\xi}. $$
Note that the fixed points of this mapping are the solutions of (\ref{fix_eq}).

In parallel with this for any fixed $\alpha \geq 0$ and $p\geq 1$ consider the following equation in $\xi$:
\begin{equation}\label{alpha_p_eq}
1 = \left[1-z\frac{\alpha+2p-1}{p}\right] \xi - z^2 \frac{(p-1)(\alpha+p-1)}{p^2}\xi^2\,.
\end{equation}
Denote by $\zeta_p^{(\alpha)}$ the largest nonnegative $z$ value, for which both roots of this second-order equation are
non-negative, i.e.
\[
\zeta_p^{(\alpha)} = \sup \left\{ z \, \mid \, z\frac{\alpha + 2p -1}{p} \leq 1,\quad
4z^2\frac{(p-1)(\alpha+p-1)}{p^2} \leq \left(1-z\frac{\alpha+2p-1}{p}\right)^2  \right\}
\]
Short calculation shows that $\zeta_p^{(\alpha)}= \left(a_p^{(\alpha)}+2\sqrt{b_p^{(\alpha)}}\right)^{-1}$, 
where $a_p^{(\alpha)}=\frac{\alpha+2p-1}{p}$ and $b_p^{(\alpha)}=\frac{(p-1)(\alpha+p-1)}{p^2}$.

Now for $0 \leq z < \zeta_p^{(\alpha)}$ define the map
$\xi \ra \eta_p^{(\alpha)}(\xi, z)$ as the solution to
$$ 1=\left[1-z\frac{\alpha+2p-1}{p}\right]\eta_p^{(\alpha)}(\xi, z)-z^2\frac{(p-1)(\alpha+p-1)}{p^2}\eta^{(\alpha)}_p(\xi, z)\xi. $$
Thus
$$ \eta_p^{(\alpha)}(\xi, z)=\frac{1}{1-z\frac{\alpha+2p-1}{p}-z^2\frac{(p-1)(\alpha+p-1)}{p^2} \xi} $$
Observe that $\eta_p^{(\alpha_p)}(\xi,z)\xrightarrow[p\ra\infty]{} \eta_c(\xi,z) \ \ \ \forall (\xi,z)\in\R^2$. \\
For the small positive values of $\xi$  the functions $\eta_c(\xi, z)$ and $\eta_p^{(\alpha)}(\xi, z)$  are increasing.
Let us denote by
$g_{p-}^{(\alpha)}(z)$ the smaller fixed point of the mapping $\eta_p^{(\alpha)}(\xi,z)$ and observe that for $0\leq z < \zeta_p^{(\alpha)}$ the inequality
$\eta_{p}^{(\alpha)}(0, z) > 0$ holds true, thus for $0 \leq \xi < g_{p-}^{(\alpha)}(z)$ we have that
\[
\xi < \eta_p^{(\alpha)}(\xi, z) \leq  g_{p-}^{(\alpha)}(z)\,.
\]
We are going to prove by induction on $p$ that for any fixed $\alpha \geq 0$ and $0\leq z < \zeta_p^{(\alpha)}$ the inequality
\begin{equation}\label{ind_stat}
f_p^{(\alpha)}(z) \leq g_{p-}^{(\alpha)}(z)
\end{equation}
holds true. It is easy to check that for $p=1$ we have that $\zeta_{1}^{(\alpha)} = \frac{1}{\alpha+1}$ and
\[
g_{1-}^{(\alpha)}(z) = f_1^{(\alpha)} (z) = \frac{1}{1-z(\alpha+1)}\,.
\]

On the other hand straightforward calculation gives that if $0\leq z < \zeta_p^{(\alpha)}$ then
$z\frac{p-1}{p} < \zeta_{p-1}^{(\alpha)}$ thus using the induction hypothesis for $p-1$ we obtain that
\begin{equation}\label{ind_hyp}
f_{p-1}^{(\alpha)}\left(z\frac{p-1}{p}\right) \leq g_{(p-1)-}^{(\alpha)}\left(z\frac{p-1}{p}\right)\,.
\end{equation}
The latter one is the smaller fixed point of the mapping
\[
\eta_{p-1}^{(\alpha)}(\ \cdot\ , z\frac{p-1}{p}):\xi \ra \frac{1}{1-z\frac{p-1}{p}\frac{\alpha+2p-3}{p-1}-z^2\frac{(p-1)^2}{p^2}\frac{(p-2)(\alpha+p-2)}{(p-1)^2} \xi}
\]
On the other hand
\begin{align}
\eta_{p-1}^{(\alpha)}(\xi,z\frac{p-1}{p}) &= \frac{1}{1-z\frac{p-1}{p}\frac{\alpha+2p-3}{p-1}-z^2\frac{(p-1)^2}{p^2}\frac{(p-2)(\alpha+p-2)}{(p-1)^2} \xi} \\
& = \frac{1}{1-z\frac{\alpha+2p-3}{p}-z^2\frac{(p-2)(\alpha+p-2)}{p^2} \xi} \\
&\leq \frac{1}{1-z\frac{\alpha+2p-1}{p}-z^2\frac{(p-1)(\alpha+p-1)}{p^2} \xi}  = \eta_p^{(\alpha)}(\xi, z)\,,
\end{align}
proving that
\begin{equation}\label{major}
g_{(p-1)-}^{(\alpha)} \left(z\frac{p-1}{p}\right) \leq g_{p-}^{(\alpha)}(z)\,.
\end{equation}
But equation (\ref{f_rec}) implies that for $\xi = f_{p-1}^{(\alpha)}\left(z\frac{p-1}{p}\right)$
\[
\eta_p^{(\alpha)} (\xi, z) = f_p^{(\alpha)}(z)\,.
\]
Comparing (\ref{ind_hyp}) and (\ref{major}) we obtain that for $0 \leq z < \zeta_p^{(\alpha)}$
\[
f_p^{(\alpha)}(z) \leq g_{p-}^{(\alpha)}(z) 
\]
proving the induction step.

 Since $a_p^{(\alpha_p)}\ra c+2$,  $b_p^{(\alpha_p)}\ra c+1$ and so 
$\zeta_p^{(\alpha_p)} \ra \frac{1}{(\sqrt{c+1}+1)^2}$,  if $\alpha_p/p\ra c$ as $p \ra \infty$ 
the following implication holds for large enough $p$:
$$ \left[0,  \frac{1}{2(\sqrt{c+1}+1)^2}\right) \subset \left[0, \zeta_p^{(\alpha_p)}\right) $$

Hence for $0\leq z<\frac{1}{2(\sqrt{c+1}+1)^2}$ we have that 
$ g_{p-}^{(\alpha_p)}(z)\xrightarrow[p\ra \infty]{} f_{c-}(z) $ as $p\ra \infty$, thus inequality (\ref{ind_stat}) implies 
that 
\[
\lim_pf_p^{(\alpha_p)}(z)=f_{c-}(z).
\]

Now let $\mathfrak{M}_c(z)=\lim_p\frac1p \mathcal{M}_p^{(\alpha_p)}\left(\frac z p \right). $ According to (\ref{f_M_kapcs}) 
we have that
 \begin{eqnarray}
    \mathfrak{M}_c(z)=(c+1)zf_{c-}(z)+1 \nonumber
  \end{eqnarray}
from which one has
\begin{align}
    \mathfrak{M}_c(z)= \frac{1-cz-\sqrt{[1-(c+2)z]^2-4(c+1)z^2}}{2z} = \frac{1-cz-\sqrt{(1-cz)^2-4z}}{2z}.
\end{align}

\item[2.] Consider now the case when $\alpha_p\in\{-p+1,\ldots,-1\}$ for all $p$ in such a way that $\lim \alpha_p / p = c$ exists and
$c > -1$. Obviously this implies $c \leq 0$.

In this case the recursion (\ref{Lag_recursion}) is still valid, but orthogonality (with respect to 
$z^{\alpha_p}e^{-z}\mathbf{1}_{[0,\infty)}(z)dz$) cannot be assured. According to (\ref{monic_lag_coeff}) one has that
$$ \ell_p^{(\alpha_p)}(z)=z^{-\alpha}\ell_{p+\alpha_p}^{(-\alpha_p)}(z) $$
and so
$$\L_p^{(\alpha_p)}(z)=z^p \ell_p^{(\alpha_p)}(1/z)= z^p z^{\alpha_p}
\ell_{p+\alpha_p}^{(-\alpha_p)}(1/z)=z^{p+\alpha_p}\ell_{p+\alpha_p}^{(-\alpha_p)}(1/z)=\L_{p+\alpha_p}^{(-\alpha_p)}(z) $$
hence
$$ \mathcal{M}_p^{(\alpha_p)}(z)=-z\frac{\frac{d}{dz}\L_{p+\alpha_p}^{(-\alpha_p)}(z)}
{\L_{p+ \alpha_p}^{(-\alpha_p)}(z)}+p=\mathcal{M}_{p+\alpha_p}^{(-\alpha_p)}(z)-\alpha_p $$
therefore we immediately get that $\mathcal{M}_p^{(\alpha_p)}(z)$ is monotonically increasing convex function if $z\geq 0$ and
$$ \frac 1p \mathcal{M}_p^{(\alpha_p)}\left( \frac z p \right)=\frac {p+\alpha_p} {p}\frac 1 {p+\alpha_p}
\mathcal{M}_{p+\alpha_p}^{(-\alpha_p)}\left(\frac {p+\alpha_p}{p} \frac z {p+\alpha_p}  \right)-\frac{\alpha_p} p. $$
Due to $\alpha_p/p\ra c $ we have $-\frac{\alpha_p}{p+\alpha_p} \ra -\frac{c}{c+1}$
and
$$ \frac 1 {p+\alpha_p} \mathcal{M}_{p+\alpha_p}^{(-\alpha_p)}\left(\frac z {p+\alpha_p} \right)
\ra 
\mathfrak{M}_{-\frac{c}{c+1}}(z) .$$

Since $f_{p+\alpha_p}^{(-\alpha_p)} $ is a sequence with uniformly bounded derivatives ( according to (\ref{int-bd}) ) one has that for $0\leq z < \left[2(\sqrt{c+1}+1)^2\right]^{-1}$
\begin{eqnarray} \mathfrak{M}_c(z)=(c+1)\mathfrak{M}_{\frac{-c}{c+1}}((c+1)z)-c \nonumber \end{eqnarray}
implying that
$$ \mathfrak{M}_{c}(z)=\frac{1-cz-\sqrt{(1-cz)^2-4z}}{2z}.$$
\end{itemize}

\noindent Since $\mathfrak{M}_c(z)$ coincides with the generating function of the moments of $\mu_c$, i.e.
$$ \mathfrak{M}_c(z)=\sum_{k\geq 0 } \int x^k d\mu_c(x) \cdot z^k\,, \ \  for \ \ z\in[0,\frac{1}{2}(\sqrt{c+1}+1)^{-2}) $$ 
and $\mu_c$ is fully determined by its moments we have that the weak limit of the empirical measure of the normalized zeros of $\ell_p^{(\alpha_p)}(z)$ is $\mu_c$.
Theorem \ref{MP_weaklimit} is hereby proved.
\end{proof}
\begin{corollary}
Theorem \ref{MP_weaklimit} also implies the convergence of the moments of the empirical distribution of the normalized roots of $\ell_p^{(\alpha_p)}$. In other words if $m_p^{(\alpha_p)}$ denotes the empirical distribution of the normalized roots of $\ell_p^{(\alpha_p)}$, then
$$ \int x^k dm_p^{(\alpha_p)}(x) \xrightarrow[p\ra\infty]{} \int x^k d\mu_c(x)=\sum_{j=1}^k \frac1k \binom{k}{j}\binom{k}{j-1} (c+1)^{k-j+1} \ \ \ \forall k\geq0 $$
when $\frac{\alpha_p}{p}\ra c$. 
\end{corollary}

\section{The sum of the $k^{th}$ power of the roots of $L_{p}^{(\alpha)}$}

The following theorem shows that the connection between the root distribution of the Laguerre-polynomials and the Marchenko-Pastur distribution is not only an asymptotic connection but in a "dominating way" it holds for large enough $p$ values, as well.

\begin{theorem}\label{highest}
    Let $p\in\N$,  $M^{(\alpha)}_p(k):=\sum_{j=1}^p{\xi^k_{p,j}}$, where $0\leq \xi_{p,1}^{(\alpha)}<\xi_{2,p}^{(\alpha)}<\ldots<\xi_{p,p}^{(\alpha)}<\infty$
denotes the roots of $L_p^{(\alpha)}$. Then for $\alpha\in\R$, $\alpha+p>k-1$ one has
    $$ M^{(\alpha)}_p(k)= \sum_{j=1}^{ k} \frac1k\binom k j \binom {k}{j-1} p^{j} (\alpha+p)^{k-j+1}+f(\alpha+p,p) $$
    where $f$ is a polynomial in two variables with $\deg f\leq k$.\\
    In case $\alpha+p\leq k-1$ one has that the coefficient of the dominating term in $M^{(\alpha)}_p(k)$ is less than or equal to the quantity above.
\end{theorem}
\begin{proof}
    Let us consider the Newton identities $\sum_{j=0}^{k-1} M^{(\alpha)}_n(k-j)a_{p-j}=-ka_{n-k},$ where $a_{p-j}$ denote the $j^{th}$ coefficient of $\ell_p^{(\alpha)}(x)$.
    It is known that $a_j=(-1)^{p+j} p!\binom{\alpha+p}{p-j} \frac{1}{j!}$ (see e.g. \cite{sze39}), hence
    $$ a_{p-j}=(-1)^{j}\frac{(\alpha+p)_j(p)_j}{j!}. $$
    Writing the Newton identities in matrix form we obtain that
\begin{align}\label{Newtonmatrix}
    \begin{bmatrix}
        1 & 0 & 0 & \ldots & 0 \\
        a_{p-1} & 1 & 0 & \ldots &0 \\
        a_{p-2} & a_{p-1} & 1 &\ldots &0 \\
        & & &  \ddots  &  \\
        a_{p-(k-1)} & a_{p-(k-2)} & a_{p-(k-3)} & \ldots & 1
    \end{bmatrix}\begin{bmatrix}
        M^{(\alpha)}_p(1) \\ M^{(\alpha)}_p(2) \\ M^{(\alpha)}_p(3) \\ \vdots  \\ M^{(\alpha)}_p(k)
    \end{bmatrix}= \begin{bmatrix}
        -a_{p-1} \\ -2a_{p-2} \\ -3a_{p-3} \\ \vdots  \\ -ka_{p-k}
    \end{bmatrix}\,.
\end{align}
Thus
\begin{align*}
    M^{(\alpha)}_p(k)=\det \begin{bmatrix}
        1 & 0 & 0 &\ldots & -a_{p-1} \\
        a_{p-1} & 1 & 0 & \ldots & -2a_{p-2} \\
        a_{p-2} & a_{p-1} & 1 & \ldots & -3a_{p-3} \\
        & & & \ddots & \\
        a_{p-(k-1)} & a_{p-(k-2)} & a_{p-(k-3)} & \ldots & -ka_{p-k}
    \end{bmatrix}
\end{align*}
according to Cramer's rule and the fact that the determinant of the matrix in (\ref{Newtonmatrix}) is $1$.
In general, let us introduce
the following notation:
$$ A(k,l):=\det \begin{bmatrix}
    1 & 0 &  \ldots & (\alpha+p)_lp \\
    -(\alpha+p)p & 1 &  \ldots & -2\frac{(\alpha+p)_{l+1}(p)_{2}}{2} \\
    &  & \ddots & \\
\frac{(-1)^{k-1}(\alpha+p)_{(k-1)}(p)_{k-1}}{(k-1)!} &\frac{ (-1)^{k-2}(\alpha+p)_{k-2}(p)_{k-2}}{(k-2 )!}  & \ldots & -k \frac{(-1)^k(\alpha+p)_{l+k-1} (p)_{k}}{k!}
\end{bmatrix}, $$
for $k\geq 2, l \geq 1$ and $A(1,l)=(p+\alpha)_lp$  for $l\geq 1$. With this notation $A(k,1)=M^{(\alpha)}_p(k)$ and it can be proved by induction that
for $k \geq 2$
\begin{align} \label{A_recursion}A(k,l)=\sum_{r=1}^{l}{p (\alpha+p-r)_{l-r} A(k-1,r)} +A(k-1,l+1), \end{align}
In fact, for $k\geq 3$
let us  subtract $p(\alpha+p)_l$ times the first column of the matrix in the definition of $A(k,l)$ from the last of the same.
The $j^{th}$ element of the last column obtained this way can be written as
\begin{align*}
    &-(-1)^{j}\frac{(\alpha+p)_{l+j-1}(p)_{j}}{(j-1)!}-(-1)^{j-1}\frac{(\alpha+p)_l p(\alpha+p)_{j-1}(p)_{j-1}}{(j-1)!}=\\
    &\hspace{5mm}=-(-1)^{j-1}\frac{(\alpha+p)_{j-1}(p)_{j-1}}{(j-1)!}[(\alpha+p-j+1)_l(p-j+1)-(\alpha+p)_lp]=\\
    &\hspace{5mm}=(-1)^{j-1}\frac{(\alpha+p)_{j-1}(p)_{j-1}}{(j-1)!}(j-1)\left(\sum_{r=1}^l
{(\alpha+p-j+1)_r(\alpha+p-r)_{l-r}}p+1  \right)
\end{align*}
due to
$$ \prod_{i=1}^m{c_i}-\prod_{i=1}^m{d_i}=\sum_{h=1}^{m}{\prod_{1\leq e<h}c_e(c_{h}-d_{h})\prod_{m\geq e>h}d_e}, $$
with $m=l+1$, $c_i=(\alpha+p-j-i+2)$, $d_i=(\alpha+p-i+1) $ for $1\leq i \leq l$ and $c_{l+1}=(p-j+1)$, $d_{l+1}=p$. This proves the recursion in (\ref{A_recursion}) for $k\geq 3$. On the other hand
\begin{align}
& A(2,l)=\det \begin{bmatrix}
    1 & (\alpha+p)_lp \\ -(\alpha+p)p & -(\alpha+p)_{l+1}(p)_2
\end{bmatrix}=  (\alpha+p)_lp(\alpha+p-l+lp) \nonumber \\
& = \sum_{r=1}^l p (\alpha+p -r)_{l-r} (\alpha +p)_r p + (\alpha + p)_{l+1} p \nonumber
\end{align}
proving (\ref{A_recursion}) for $k=2$.

Note that in case $k\leq p+\alpha$ we have $ (\alpha+p)_l>0  $  for $0\leq l \leq k$,
and $(p+\alpha)_l=(p+\alpha)^l+O((p+\alpha)^{l-1})$ hence the multiplier in the sum in (\ref{A_recursion}) does not change
the (positive) coefficient - nor its sign - of the highest order terms of $A(k-1,r)$.

We are going to prove that viewing $A(k,l)$ as a polynomial of the variables $p$ and $p+\alpha$ one has $\deg A(k,l)=k+l$.
 The proof goes by induction on $k$. For $k=1$ and $l$ arbitrary this is an immediate consequence of its definition.
In fact -- assuming the induction hypotesis for $k-1$ and $l$ arbitrary -- we have that
\begin{align*}\deg p(p+\alpha)_{l-r} A(k-1,r)&= k-1+r+l-r+1=k+l
\quad \hbox{for} \ 1\leq r\leq l \leq k \leq p+\alpha \\  \deg A(k-1,l+1)=k+l, \end{align*}
and using that there is no cancellation in the highest degree terms we obtain that $\deg A(k,l)=k+l$, hence we immediately get that $\deg_pM_p^{(\alpha)}(k)=k+1$. \\
Computing the leading coefficient in $p$ and $\alpha+p$ of $A(k,1)$ leads to the following graph theoretical question:
Let $G=((\Z_{\geq0})^2,\overrightarrow{E})$ be the following graph: there is a directed arrow from $(a_1,b_1)$ pointing to $(a_2,b_2)$
if and only if $a_2=a_1+1$ and $b_2\geq b_1-1$. We shall also use the word edge instead of arrow in case we are not interested in its direction.
We will call an edge $(a,b_1)\ra (a+1,b_2)$ an upward edge if $ b_2\geq b_1$, if $b_2=b_1-1$ we will refer to it as a downward edge. The height of an edge $((a,b_1),(a+1,b_2))$ is going to be defined as $b_2-b_1$, total height of a set of edges is the sum of their heights.

Let us call a path ending in $(k,l)$ for $k\geq 1,l\geq 1$ legal if it starts in the origin and after that it stays strictly above the line
$y=0$.
Since $\deg_p A(k,l)=k+l$, it can be written as $A(k,l)=\sum_{j=0}^{k+l}{a^{(k,l)}_j p^j(\alpha+p)^{k+l-j}}+\text{L.O.T.}$,
for some $a_j^{(k, l)}$, $j=0, \dots , k+l$, where $\text{L.O.T.}$ means lower order terms. But the recursion (\ref{A_recursion})
implies that the degree of $p$ in $A(k, l)$ cannot be larger then $k$ and it is at least $1$, for any $l\geq 1$, thus
$A(k,l)=\sum_{j=1}^{k}{a^{(k,l)}_j p^j(\alpha+p)^{k+l-j}}+\text{L.O.T.}$
Using the recursion (\ref{A_recursion}) again we obtain that
\[
a_j^{(k, l)} = \sum_{h=1}^l a_{j-1}^{(k-1, h)} + a_j^{(k-1, l+1)}\,.
\]

Our claim is that $a_j^{(k,l)}$ is equal to the number of legal paths $b_j^{(k,l)}$ ending in $(k,l)$ with exactly $j$ upward edges.

For $k=1$, $ l \geq 1$ we have that $A(1,l)=p(\alpha+p)_l$ thus the highest order term is $p(\alpha+p)^l$ and so
$a_1 ^{(1,l)}= 1$, while $a_j^{(1, l)} = 0$ for $j \neq 1$ obviously coinciding with the values $b_j^{(1,l)}$, $j\geq 0, l\geq 1$
since in this case the path consists of one single upward edge.

For the induction step $k-1\mapsto k$ consider the following: Each of the legal paths ending in $(k,l)$ has to go through
exactly one of the points $(k-1,r)$ $1\leq r \leq l+1$.
A path with exactly $j$ upward edges going through the points $(k-1,r)$ for $1\leq r \leq l$ should have
$j-1$ upward edges before these points, while a path going through $(k-1,l+1)$ has $j$ upward edges before this point.
Therefore the number of legal paths ending in $(k,l)$ and having $j$ upward edges is the sum of the number of legal paths
ending in $(k-1,r)$ with $1\leq r\leq l$ with $j-1$ upward edges plus the number of legal paths ending in $(k-1,l+1)$ with $j$ upward edges.
In other words:
$$ b^{(k,l)}_j=\sum_{r=1}^l{b_{j-1}^{(k-1,r)}} + b_{j}^{(k-1,l+1)}. $$
Thus the number of legal paths satisfies the same recursion as the coefficients in the sequence $A(k, l)$. Since for $k=1$ they are equal
the induction argument gives that $a_j^{(k, l)} = b_j^{(k, l)}$ for $j=1, \dots k$, $k\geq 1$, $l\geq 1$.

Now let us turn our attention to computing the coefficients of the highest order term of
$M_p^{(\alpha)}(k)=A(k,1)=\sum_{j=1}^{k}a_j^{(k,1)}p^j(\alpha+p)^{k-j+1}+\text{L.O.T} $.
As we proved before the coefficient $a_j^{(k,1)}$ is given by the number of legal paths ending in $(k,1)$ with $j$ upward edges.
In this case there are $k-j$ downward edges with total height $-(k-j)$ hence the total height of the upward edges is $k-j+1$.
Since the length of the legal path from the origin to $(k,1)$ is $k$ there are $\binom{k}{j}$ possibilities to choose the positions
of the $j$ upward edges.
On the other hand the total height of the upward edges is $k-j+1$, and there are $\binom{k}{j-1}$ ways writing it as a sum of $j$
non-negative numbers when the sequence of the summands matters. Choosing these numbers as the heights of the upward edges
we obtain a path from the origin to $(k, 1)$ which is not necessarily legal,
since they can cross the line $y=0$. For such a given
path let $(x,y)$ denote the node of the path with the largest first coordinate such that its second coordinate is not greater
than the second coordinate of any other node of the path (i.e. the latest "global minimum" of the path).
By placing this node
with the tail of the path in the origin this new path is a legal path ending in $(k-x,1+y)$. Taking the first part of the original path
(connecting the origin with $(x,y)$ ) and gluing it to $(k-x,1+y)$ we will get a legal path ending in $(k,1)$. We will say
that two paths are equivalent if the cut-and-glue process described above results in the same legal path. The equivalence class of a
path consists of its periodic horizontal translations, so in each equivalence class there are $k$ paths. Since the cut-and-glue
process gives the same legal path for each equivalence class, thus the number of legal paths ending in $(k,1)$ having $j$
upward edges is given by $ \frac 1 k \binom{k}{j}\binom{k}{j-1} $, hence
$$M_p^{(\alpha)}(k)=\sum_{j=1}^{k} \frac{1}{k} \binom{k}{j}\binom{k}{j-1}p^{j}(\alpha+p)^{k-j+1}+\text{L.O.T} $$
and so Theorem \ref{highest} is proved.

\end{proof}

\begin{remark}If $\alpha/p=c$ with $c\in(-1,\infty)$ and $k < \alpha +p + 1$  then
\begin{equation} \label{MP_mom} \sum_{l=1}^p{(\xi_{p,l}^{(\alpha)})^k}=\sum_{j=1}^{k}  \frac1k\binom{k}{j}\binom{k}{j-1}(c+1)^{k-j+1} p^{k+1}  +f(\alpha+p,p) \end{equation}
hence we immediately get that $$ \int x^k dm_p^{(\alpha_p)}(x)\xrightarrow[p\ra\infty]{}\int x^k d\mu_c(x) $$
if $\frac{\alpha_p}{p} \ra c$ for all $k\geq 0$.

We also emphasize that even in the case when $\alpha<0$ is not an integer thus $L_p^{(\alpha)}(z)$ has complex roots with nonzero imaginary part, the limit relation above holds true. But since now the measure is not concentrated on the real line this property is not enough for the identification of the the limit measure.

\end{remark}

\bibliography{KM}
\bibliographystyle{plain}
\end{document}